\documentclass[12pt]{article}
\usepackage{amsmath,amssymb,amsfonts,amsthm}
\usepackage{eucal} 
\usepackage[dvips]{graphics}

\newcommand{\bK}{\mathsf{K}}

\newcommand{\Pp}{\mathbf{P}^2}
\newcommand{\bT}{\mathsf{T}}
\newcommand{\bA}{\mathsf{A}}

\newcommand{\cM}{M}
\newcommand{\Ms}{\mathcal{M}}
\newcommand{\cN}{\mathsf{Q}}
\newcommand{\cL}{\mathcal{L}}

\newcommand{\ba}{\mathbf{a}}
\newcommand{\bb}{\mathbf{b}}

\newcommand{\sdOm}{\deg\Omega}
\newcommand{\sw}{\mathsf{w}}

\DeclareMathOperator{\supp}{supp}

\DeclareMathOperator{\Prob}{Prob}

\newcommand{\lang}{\left\langle}
\newcommand{\rang}{\right\rangle}

\newcommand{\bS}{\mathsf{S}}

\newcommand{\Z}{\mathbb{Z}}
\newcommand{\R}{\mathbb{R}}

\newcommand{\cO}{\mathcal{O}}

\DeclareMathOperator{\Aut}{Aut}

\DeclareMathOperator{\ind}{ind}
\DeclareMathOperator{\Jac}{Jac}
\DeclareMathOperator{\Tor}{Tor}
\DeclareMathOperator{\Img}{Im}

\newtheorem{Theorem}{Theorem}
\newtheorem{Lemma}{Lemma}

\theoremstyle{definition}

\begin{document}
\title{Noncommutative geometry of random surfaces}
\author{Andrei Okounkov}
\date{} \maketitle

\section{Introduction}

\subsection{}

This paper is about a certain interaction between probability and geometry. The random objects 
involved will be random stepped surfaces spanning a given boundary in $\R^3$. Equivalently, one 
can talk about random rhombi tiling of a planar domain or random dimer coverings of certain subgraphs
of the hexagonal graph. Probabilistic questions about 
these random surfaces  will be answered in terms of a nonrandom algebraic object, geometrically 
a curve in a noncommutative plane. 

Underlying this connection is Kasteleyn's theory of planar dimers which computes all probabilities
in terms of the Green's function of a certain finite-difference operator $\bK$. It will be clear 
from our construction that the connection between finite-difference operators and noncommutative
geometry may be easily extended far beyond what we do in this paper. Our goal here, however, 
is to explain a certain phenomenon in the least possible generality and to stay as close as
possible to certain specific applications. These applications, as well as some other directions
that look promising will be discussed below.

\subsection{}

Let $\Omega$ be a simply-connected planar domain that can be tiled by rhombi as shown in 
Figure \ref{Fig1}. Well-known bijections illustrated in Figure \ref{Fig1} identify such tilings
with dimer coverings of subgraph $\Omega_6=\Omega \cap \Gamma_6$ of the hexagonal graph $\Gamma_6$
and also with \emph{stepped
surfaces} spanning given boundary. An introduction to dimers and stepped surfaces may be 
found in \cite{KenLec}.

\begin{figure}[!hbtp]
  \centering
\scalebox{0.3}{\includegraphics*[55,230][570,675]{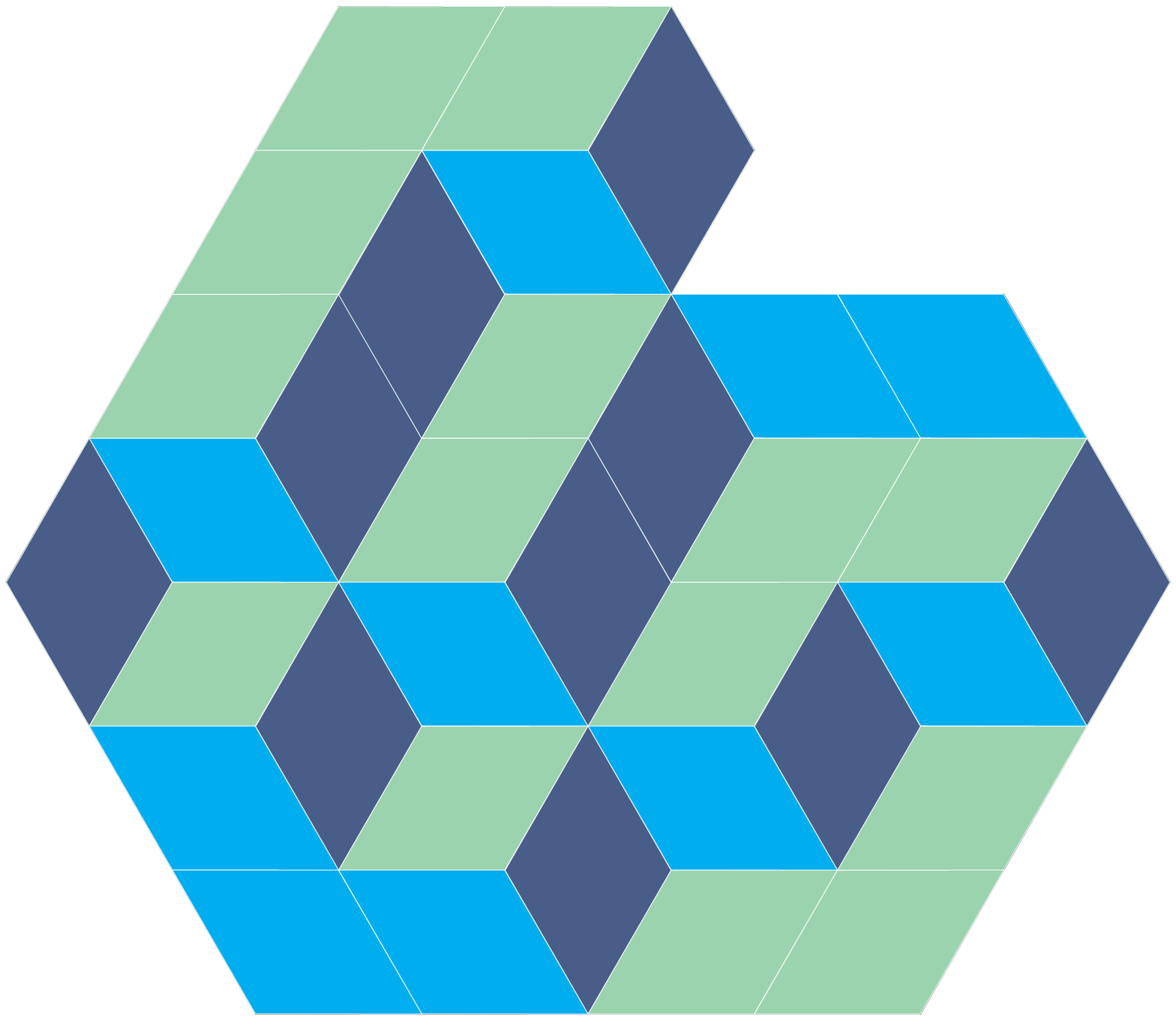}}
\qquad
\scalebox{0.3}{\includegraphics*[30,240][545,685]{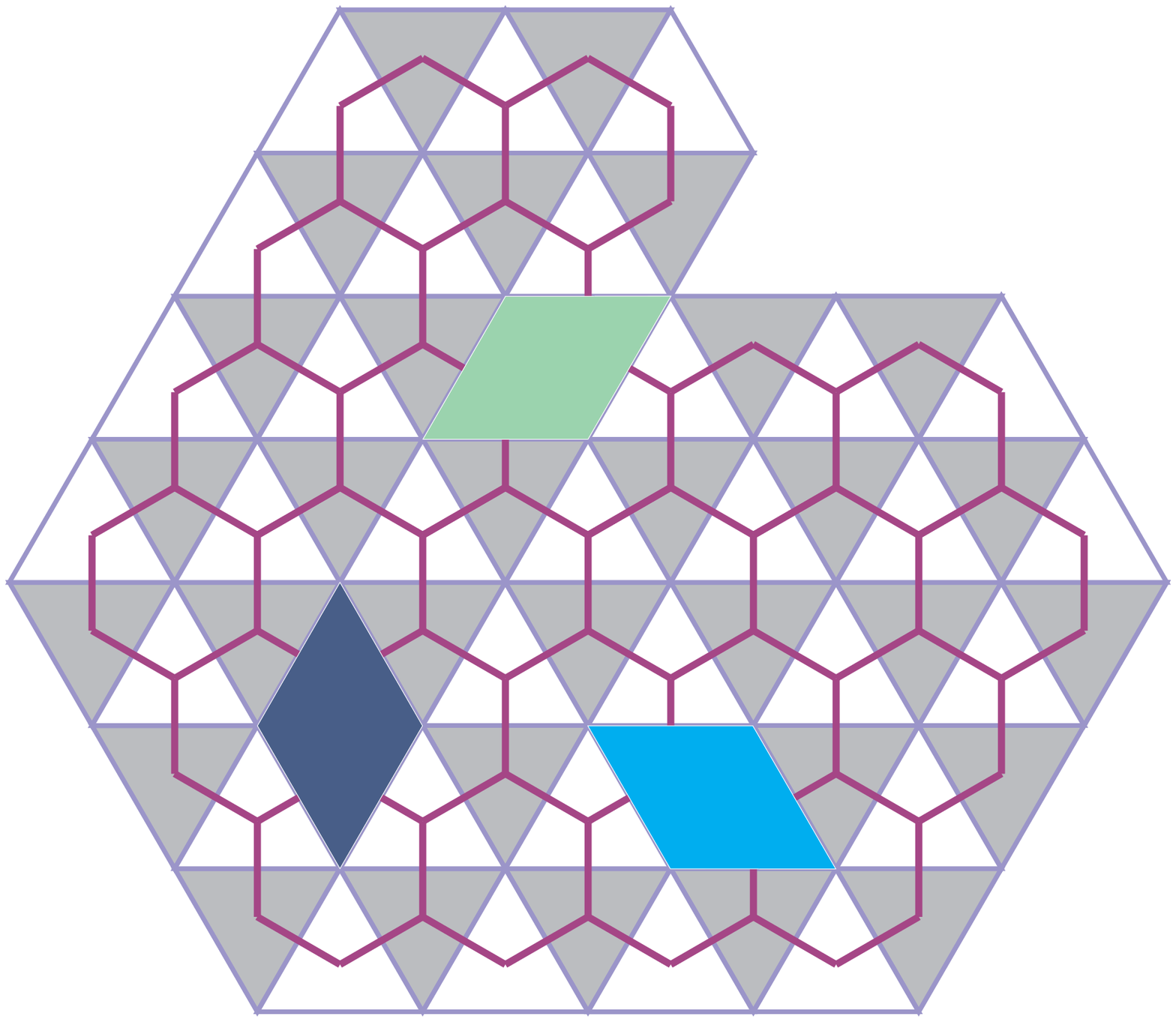}}
\caption{\small{Stepped surfaces are the same as tilings of a planar domain $\Omega$ by rhombi. Each tile 
is a union of a black and white triangle. The adjacency graph of the triangles is (a piece of) the 
$6$-gonal graph.}} 
  \label{Fig1}
\end{figure}

Stepped surfaces arise in mathematical physics in a variety of contexts: from 
simply-minded, but realistic models of interfaces (e.g.\ crystalline surfaces) to the sophisticated setting of
super-symmetric gauge and string theories (see e.g.\ \cite{Takagi,EC} for an 
introduction). 

In all applications, it is natural to weight the probability of a stepped surface $S$
by the volume $V(S)$ enclosed by it, i.e. to set
\begin{equation}
\Prob(S)\propto q^{V(S)}\,, \label{Probq}
\end{equation}
where $q>0$ is a parameter.  Note that for two surfaces $S_1$ and $S_2$ 
spanning the same boundary the difference $V(S_1)-V(S_2)$ is well-defined,
which is all that matters in \eqref{Probq}. In the crystal surface context, 
$\log q$ is the energy price for removing an atom\footnote{Note that there is a well-known ambiguity 
in reconstructing 3-dimensional surfaces from tiling, namely, one can switch the roles of convex and concave corners.
A rotation by $\pi/3$ interchanges the choices. Our conventions are fixed by \eqref{qK3}.}

\subsection{}

We will be particularly 
interested in the case when $\Omega$ grows to infinity, while keeping its shape, that is, 
the number and orientation of its boundary segments. We will refer to such domains 
as \emph{polygonal}. The study of a random stepped surface a large polygonal boundary 
(equivalently, random dimer covering, or 
random tiling) leads to interesting probabilistic questions. The nontriviality of these
questions may be appreciated by looking at  Figure \ref{Fig2}. 

\begin{figure}[!hbtp]
  \centering
\scalebox{0.35}{\includegraphics*[45,145][570,747]{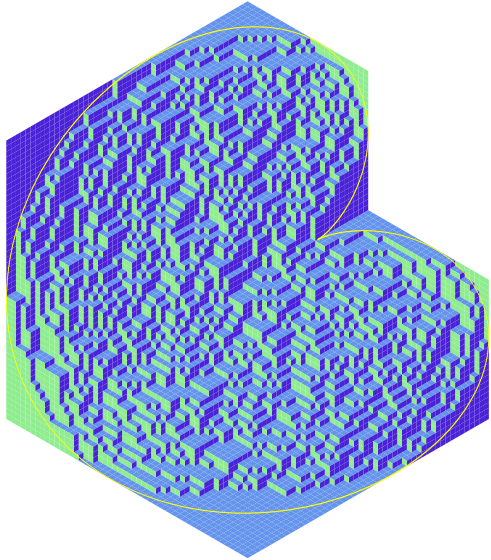}}
\caption{\small{A simulation of the limit shape formation. The curve separating the ordered regions (facets) from the disorder
is called the \emph{frozen boundary}. Here, it is a cardioid, i.e.\ the dual of 
a rational cubic.}} 
  \label{Fig2}
\end{figure}

Apparent in Figure \ref{Fig2} is a formation of a certain nonrandom \emph{limit shape}, 
in other words, as the mesh size goes to zero so does the scale of randomness. 
This is a form of the law of large numbers. The existence of a limit shape 
was proven for stepped surfaces (with arbitrary boundary conditions) by 
H.~Cohn, R.~Kenyon, and J.~Propp in \cite{CKP}. 

In \cite{Burg}, the limit shape for polygonal boundaries 
was linked to a certain plane algebraic curve $Q$. In particular, the frozen boundary, which is the curve 
separating order from disorder in Figure \ref{Fig2}, is the planar dual of $Q$ in exponential 
coordinates.

\subsection{}

The main object of this paper may be characterized, informally, as a \emph{quantization} 
of the limit shape, or of the curve $Q$ to be more specific. This quantization
 exists for finite $\Omega$, that is, before any limits are taken. In particular, it captures 
not just the limit shape but also 
the \emph{fluctuations} of our random surfaces. Viewed like this, it should not be 
surprising to see noncommutative curves appear. Note, however, that technically the
noncommutativity will be linked to the parameter $\log q$ and not to the size of 
fluctuations (as could be expected from the uncertainty principle).

\subsection{}

In principle, classical Kasteleyn's theory \cite{Kas} answers all possible questions about random 
stepped surfaces in terms of the Green's function, i.e.\ the inverse of a certain 
difference operator.  This \emph{Kasteleyn operator} $\bK$ is a
weighted adjacency matrix of the graph $\Omega_6$.

 Note that $\Gamma_6$, and hence, 
$\Omega_6$ is bipartite, that is, its vertices may be colored in two colors (black and 
white, traditionally) so that only vertices of opposite color are joined by an edge. 
This is reflected in Figure \ref{Fig3}. We will index the rows (resp.\ columns) of the 
adjacency matrix by white (resp.\ black) vertices. The nonzero matrix elements $\bK_{ij}$
should satisfy 
\begin{equation}
  q \, \bK_{21} \bK_{43} \bK_{65} = \bK_{23} \bK_{45} \bK_{61} \label{qK3}
 \end{equation} 
for each face $F$, where $v_1,\dots,v_6$ are the six vertices going around $F$ as in 
Figure \ref{Fig3}. This fixes $\bK$ uniquely up to a certain gauge transformation, namely 
left and right multiplication by a diagonal matrix. 
\begin{figure}[!hbtp]
  \centering
\scalebox{0.75}{\includegraphics*[320,350][475,515]{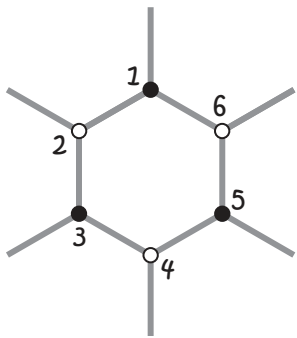}}
\caption{\small{Each hexagon face carries $\log q$ units of magnetic flux.}} 
  \label{Fig3}
\end{figure}

The goal of this paper may be informally described as looking for some hidden structures
in the inverse matrix $\bK^{-1}$. Certain structures in $\bK^{-1}$ are plain to see: by definition,  
the entries of $\bK^{-1}$ satisfy a finite-difference equation in each index, namely
$\bK \, \bK^{-1} = \bK^{-1} \, \bK = 1$. 

Our main claim is that for polygonal domains $\Omega$ the entries of $\bK^{-1}$ satisfy 
\emph{additional} finite-difference equations. The degree of these additional 
equations is determined by the shape of $\Omega$, that is, by the number of 
boundary segments, and not by the size of $\Omega$. This is crucial from 
the probabilistic viewpoint. 

\subsection{}

The noncommutative geometry of the title provides a natural language to state and 
study these additional equations. 

The origin of the noncommutativity may be traced to \eqref{qK3}. For $q=1$, the adjacency 
matrix of $\Gamma_6$ is an obvious solution and this solution is translation-invariant, i.e.\
commutes the the subgroup $\Z^2\subset \Aut(\Gamma_6)$ acting by (bipartition-preserving) 
translations. 

For $q\ne 1$, the translation-invariant equation \eqref{qK3} has no translation-invariant solutions, 
which means that the Kasteleyn operator $\bK$ now commutes with \emph{magnetic translations}, i.e.\
translations followed by a gauge transformation. In turn, magnetic translations commute only up to 
a factor  (whose logarithm is proportional to the area of the parallelogram spanned by the 
translation vectors). They form, in other words, an algebra known as the quantum 2-torus. 

The considerations so far apply only to the whole 6-gonal graph $\Gamma_6$, that is, in the 
absence of any boundaries. Somewhat remarkably, however, a certain framework may be established 
in which the Kasteleyn operator and the commuting magnetic translations act in a way 
compatible with polygonal boundaries. This involves compactifying the quantum 2-torus to a 
noncommutative plane, meaning that one introduces a certain graded algebra $\bA$, 
a deformation of the ring of polynomials in $x_1,x_2,x_3$, such that the quantum
$2$-torus is the degree $0$ part of $\bA\left[(x_1 x_2 x_3)^{-1}\right]$. 

\subsection{}

For any fixed white vertex $\sw\in\Omega_6$, the  action of magnetic translation on 
$\bK^{-1}(\,\cdot\,,\sw)$, that is, on the corresponding column of $\bK^{-1}$, 
yields a graded $\bA$-module $\cN^\sw$.  
The additional equations satisfied by $\bK^{-1}$ will be reflected in the fact that $\cN^\sw$ is a \emph{torsion}
module. 

For different $\sw$, the modules $\cN^\sw$ share the same fundamental features. In fact, 
there is a canonical submodule $\cN$ in all of them that depends on $\Omega$ only and 
captures the essential information.

\subsection{}

The construction of the module $\cN$ and the study of its basic properties will occupy the 
bulk of the present paper. While the definition of $\cN$ involves 
nothing beyond elementary combinatorics and linear algebra, we will find 
the resulting object 
has a certain depth and complexity.

The degrees of its generators and relations 
(and hence the degrees of the additional 
equations satisfied by $\bK^{-1}$) are determined by the combinatorics of 
the domain $\Omega$ only, see Theorem \ref{Th1}. 
On the other hand, the explicit form of these relations depends of $q$ and 
the geometry of $\Omega$ in a rather intricate fashion.

\subsection{}

This is where a geometric way of thinking about such modules, pioneered by M.~Artin and 
his collaborators, becomes essential (see e.g.\ \cite{SV} for an introduction). While its a matter of definitions to associate to 
$\cN$ a sheaf on the noncommutative plane, the geometric intuition thus gained 
is very valuable. 

In the first place, this is what allows us to view $\cN$ 
as a quantization of the limit shape $Q$. In fact, quantization requires additional degrees of 
freedom, parametrized by line bundles (and more general rank $1$ sheaves) on $Q$. These may be compared to complex phases in quantum 
mechanics. 

Certain features of $Q$, such as its points of intersection with the coordinate axes of $\Pp$
have a direct quantum analog satisfied by $\cN$, see Section \ref{sbdcN}. Others (such as 
the geometric genus) are harder to generalize, see \cite{KO}. 

\subsection{}

Because $\Omega$ is a purely combinatorial object, we can modify it by simply moving 
boundary segments in and out. We prove in Section \ref{srenorm} that these transformations
act on $\cN$ by what may be called a noncommutative shift on the Jacobian. In particular, this describes what happens to $\cN$ under rescaling of $\Omega$.

\subsection{Acknowledgments}

The author's present understanding of the subject took some time to develop and 
I have a large number of people to thank for stimulating and insightful discussions along the way. In particular, I thank D.~Eisenbud, V.~Ginzburg, A.\ J.\ de Jong, 
R.~Kenyon, I.~Krichever, 
D.~Maulik, N.~Nekrasov, E.~Rains, N.~Reshetikhin, and Y.~Soibelman. Parts of this 
work grew into joint research projects \cite{KO,RO}.

I had an opportunity 
to lecture on the subject 
on a number of occasions, in particular during the Aisenstadt lectures
at the Universit\'e de Monr\'eal, 
T.~Wolff lectures at Caltech, 
Eilenberg lectures at the Columbia
University, and Milliman lectures at the University of Washington. 
I am very grateful to the participants of these lectures for 
their involvement and feedback. I very much 
thank these institutions and the Institut
des Hautes \'Etudes Scientifiques for their warm hospitality during my work 
on this paper.

\section{The quantum limit shape}

\subsection{}

Let $\bA$ be the algebra generated by $x_1,x_2,x_3$ subject to the 
relations 
\begin{equation}
  \label{commR}
    x_{j} x_i = q_{ij} \, x_i x_{j}\,,  
\end{equation}
where $q_{ii}=1$ and $q_{ij}=q_{ji}^{-1}$. This is a basic example of a noncommutative projective plane $\Pp$, 
see e.g.\ \cite{SV} for an introduction. Adding the inverses $x_i^{-1}$ 
of the generators, we obtain a larger algebra $\bT$ known as a 
noncommutative 3-torus. 

Both $\bA$ and $\bT$ are graded by the 
total degree in all three generators.
Let $\bT_d$, where $d\in \Z$ is arbitrary, be one of the graded components. The monomials
$$
x^a = x_1^{a_1} x_2^{a_2} x_3^{a_3} \in \bT_d\,,
$$ 
obviously correspond to lattice points $a\in\Z^3$ lying in the 
plane $a_1+a_2+a_3 = d$. Their projection in the $(1,1,1)$ direction, that 
is, their image in $\R^3/\R\cdot(1,1,1)\cong \R^2$ may be identified with 
the black vertices of $\Gamma_6$. In the same fashion, 
the monomials in $\bT_{d+1}$ are put into bijection with the white vertices
of $\Gamma_6$. This is illustrated for $d=4$ in Figure \ref{Fig4}. 

\begin{figure}[!hbtp]
  \centering
\scalebox{0.4}{\includegraphics*[123,95][478,450]{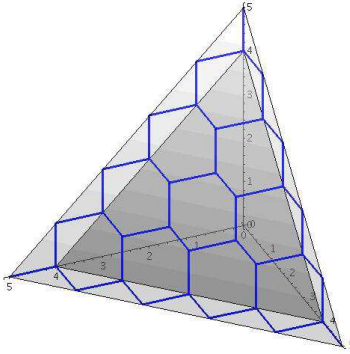}}
\vspace{0 cm}
\caption{\small{Black and white vertices may be identified with points of $\Z^3$ lying on two parallel 
hyperplanes and thus with monomials of two successive degrees. Here we only plot points in $\R_{\ge 0}^3$
which correspond to polynomials in the $x_i$'s.}} 
  \label{Fig4}
\end{figure}

\subsection{}

Consider the operator 
$$
\bK: \bT_d \to \bT_{d+1}
$$
given by right multiplication by $x_1 +x_2 +x_3 \in W_1$, that is, 
$$
\bK f = f \cdot (x_1 +x_2 +x_3) \,.
$$
One easily checks from the commutation relations \eqref{commR} that, indeed, in the 
basis of monomials, this is 
a $q$-weighted Kasteleyn operator for $\Gamma_6$ as above, with 
\begin{equation}
  \label{qqq}
    q = q_{12} \, q_{23} \, q_{31} \,.  
\end{equation}
Note that while there is no canonical ordering of the variables and, hence, no 
canonical normalization of a monomial, the lines spanned by monomials in $\bT$ 
are well defined. This is all we need since we only care about $\bK$ modulo 
gauge transformations. 


\subsection{The module $\cM$}

\subsubsection{}

Monomials are normal in $\bT$, that is, $\bA \, x^c = x^c \, \bA$ is naturally a 
$\bA$-bimodule. Assuming $\deg x^c \le 1$, monomials in $(\bA \, x^c)_1$ form a triangle with vertices 
$$
x_1^{c_1} x_2^{c_2} x_3^{-c_1-c_2+1}\,, x_1^{c_1} x_{2}^{-c_1-c_3+1} x_{3}^{c_3}\,,
x_1^{-c_2-c_3+1} x_{2}^{c_2} x_{3}^{c_3} \, \in \bT_1 \,. 
$$
We denote this equilateral triangle by $\Delta(c)$ and call it the \emph{support} of $(\bA \, x^c)_1$\,.

\subsubsection{}

We think of black and white vertices in $\Omega_6$ as monomials in $\bT_0$ and $\bT_1$, 
respectively. Write $\Omega$ as a set-theoretic combination of triangles 
\begin{equation}
  \label{OmTr}
  \Omega= \bigcup_i \Delta(\ba_i) \setminus \bigcup_j \Delta(\bb_j)\,, \quad \ba_i,\bb_j\in \Z^3 \,,
\end{equation}
as in Figure \ref{Fig5}. The inclusion of triangles induces the inclusion of graded $\bA$-bimodules 
\begin{equation}
  \label{inclusion}
  \bigoplus \bA \, x^{\bb_i} \subset \bigoplus \bA \, x^{\ba_i} \,.
\end{equation}
We denote by $\cM$ the quotient $\bA$-bimodule in \eqref{inclusion}. Define an endomorphism $\bK$ of $\cM$ as 
the right multiplication by $x_1+x_2+x_3 \in \bA_1$. This is a map of left $\bA$-modules. 
By construction, the operator 
$$
\bK_0: \cM_0 \to \cM_1 
$$
is the Kasteleyn operator for $\Omega_6$.

\begin{figure}[!hbtp]
  \centering
\scalebox{0.3}{\includegraphics{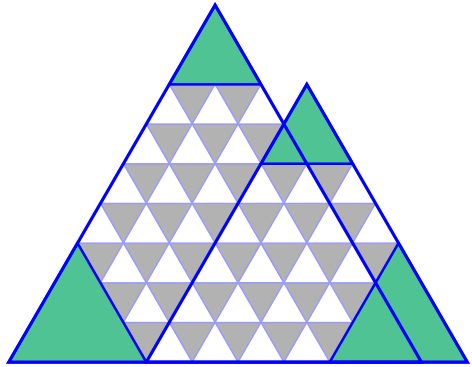}}
\vspace{0 cm}
\caption{\small{The domain $\Omega$ may be represented as a difference of unions of triangles. Here, it is the union of 2 blue triangles minus 
the union of 4 green ones.}}
  \label{Fig5}
\end{figure}

\subsubsection{}

The idea of \emph{stable range} will play an important part in this paper. By definition, 
a number $i$ is in the stable range for $\cM$ if the domain 
$$
\Omega(i) = \supp \cM_{i+1}
$$
has the 
same combinatorics as $\Omega=\Omega(0)$. In particular, the length of the shortest 
white boundary of $\Omega$, i.e.\ a boundary formed by white triangles, gives an 
upper bound on the stable range. 

More generally, refining \eqref{OmTr}, $\cM$ may be 
given a resolution 
\begin{equation}\label{resM}
0\to F_3 \to F_2 \to F_1 \to F_0 \to \cM \to 0 
\end{equation}
by modules of the form $\bigoplus \bA \, x^{\ba_i}$, where the maps are the natural inclusions. 
Here $F_0=\bigoplus \bA \, x^{\ba_i}$ corresponds to the triangles $\Delta(\ba_i)$ in \eqref{OmTr}. The intersections 
among $\bA \, x^{\ba_i}$ 
together with $\Delta(\bb_i)$ contribute to relations $F_1$, and so on.
The module $F_0$ doesn't have generators of positive degree,  but the other $F_i$'s do. 
The stable range is until the first such generator appears. 

Note that the stable range scales linearly with $\Omega$, meaning that it scales like 
the inverse mesh size in the probabilistic setup. Throughout the paper, we will only 
be interested in what happens in the stable range.

\subsubsection{}

Our next goal is to compute the Hilbert function of $\cM_d$ in the stable range. 
For that, we will make a genericity assumption that 
all 3 boundary slopes appear in the boundary $\partial \Omega$ of $\Omega$ in a cyclic order. 
We will denote by $\deg\Omega$ the number 
of times $\partial \Omega$ cycles through the three slopes. For example, $\deg\Omega=3$
in Figure \ref{Fig6}. 

\begin{Lemma}\label{lemdimM} For $d\ge 0$ in the stable range, we have 
$$
\dim \cM_d = \dim \cM_0 - \deg\Omega \, \frac{d(d-1)}{2} - d \, \ind \bK_0 \,.
$$
\end{Lemma}

Here, of course, the index of $\bK_0$ equals zero, but it will not vanish in 
the generalizations considered below. 
 
\begin{proof}
Consider the map
\begin{equation}
 \cM_d \owns f \mapsto x_3 f \in \cM_{d+1}\,. \label{mapx3}
\end{equation} 
The kernel and cokernel of this map are formed by functions
supported on horizontal strips of the form shown in Figure \ref{Fig6}. More precisely,
white horizontal boundaries correspond to the cokernel, the black ones --- to the kernel.

\begin{figure}[!hbtp]
  \centering
\vspace{0.33 cm}
\scalebox{0.28}{\includegraphics*[32,240][545,438]{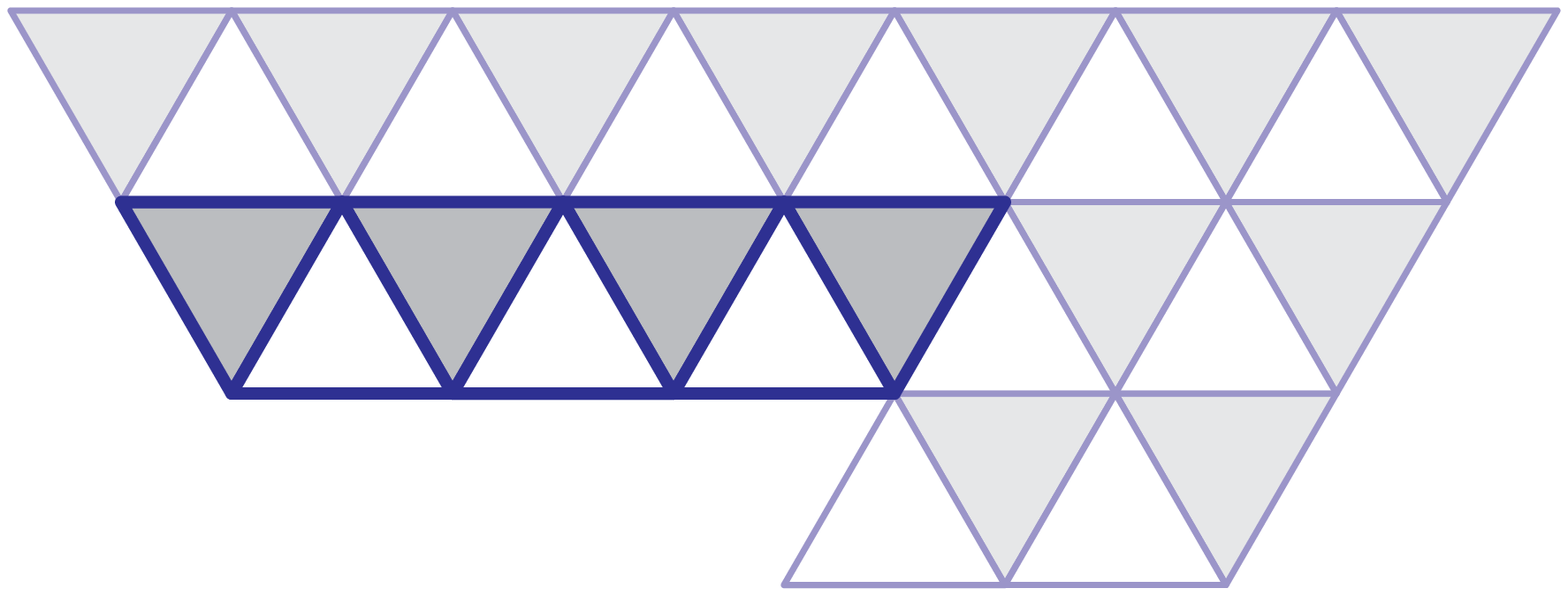}}
\quad
\scalebox{0.28}{\includegraphics*[15,355][560,548]{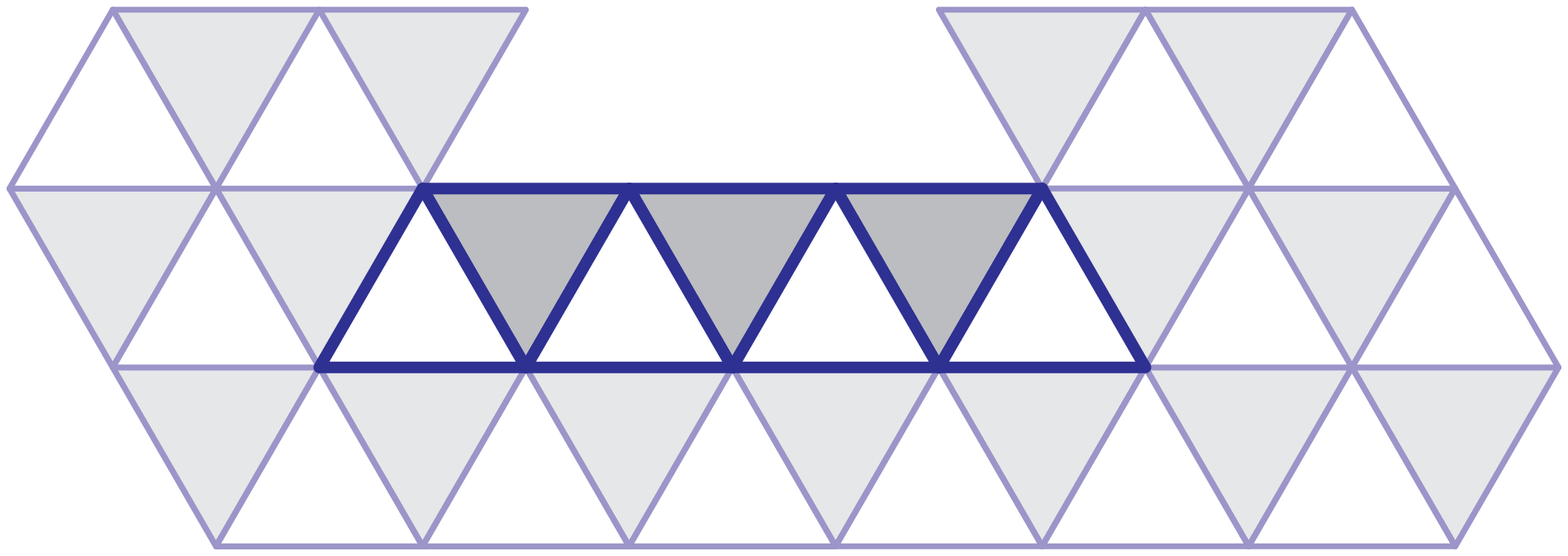}}
\vspace{0.33 cm}
\caption{\small{A white (left) and black (right) boundary strips. Note that they always have the same shape provided 
the boundary of $\Omega$ cycles through the 3 directions.}}
  \label{Fig6}
\end{figure}

Note that white boundaries of $\supp \cM_i$ shrink with $i$, while the black 
ones expand. Thus each of the $\deg\Omega$ horizontal boundaries contributes 
$-1$ to the second derivative of $\dim \cM_i$ with respect to $i$. We conclude 
$$
\dim \cM_d = - \deg\Omega \, \frac{d^2}{2} + \dots\,,\,.
$$
Here dots stand for a polynomial in $d$ of degree $\le 1$, which is uniquely fixed by 
its evaluation at $d=0,1$. 
\end{proof}

\subsubsection{}
\begin{Lemma}\label{lsurK} 
The operator $\bK: \cM_d \to \cM_{d+1}$ is surjective for $d\ge 0$ in the stable range and 
$q>0$ or $q$ generic. 
\end{Lemma}

\begin{proof}
For $d=0$, by Kasteleyn's theorem, $\det \bK$ 
determinant gives the $q$-weighted count of stepped surfaces, hence nonzero for $q>0$. 
For $d>0$, we use the commuting map \eqref{mapx3} and induction. By induction, it 
suffices to check the surjectivity of $\bK$ on a white boundary strip, as in 
Figure \ref{Fig6}, which is immediate. 
\end{proof}

\subsection{The module $\cN$ and the inverse Kasteleyn matrix}

\subsubsection{}

We denote by $\cN$ the kernel of $\bK$ acting on $\cM$. This is 
a graded left $\bA$-module. To see why this may be a useful definition, 
let us generalize the construction slightly. 

\subsubsection{}

Let $\sw$ be a white vertex of $\Omega_6$ corresponding to a monomial $x^\sw$. 
Let $\Omega^\sw$ be obtained from $\Omega$ by removing the corresponding 
white triangle and set
$$
\cM^\sw = \cM \big/ \, \bA \, x^\sw \,.
$$
The conclusions of Lemmas \ref{lemdimM} and \ref{lsurK} continue to hold 
for $\cM^\sw$, with the obvious modification that 
$$
\deg \Omega^\sw = \deg \Omega + 1 \,, \quad \ind \bK^\sw_0 = 1 \,.
$$
In contrast to $\cN$,  we have $\cN^\sw_0\ne 0$.  Indeed, by 
construction,  $\cN^\sw_0$ is spanned by the corresponding column $\bK^{-1}\, x^\sw$
of the inverse Kasteleyn matrix. 

Via the left $\bA$-module structure on $\cN^\sw$, the algebra $\bA$ acts on 
the columns of the inverse Kasteleyn matrix by difference operators.
To see that a nonzero difference operator from $\bA$ must annihilate $\cN^\sw_0$, 
it will suffice to compute the Hilbert function of $\cN^\sw$. 

\subsubsection{}

{}From Lemmas \ref{lemdimM} and \ref{lsurK} we have, in the stable range, 
\begin{equation}
 \dim \cN_d = d \, \deg \Omega  + \ind \bK_0 = d \deg \Omega \label{dimNd} 
\end{equation}
and, similarly, 
\begin{equation}
 \dim \cN^\sw_d = (\deg \Omega +1) \, d + 1 \label{dimNwd} \,. 
\end{equation}
Since this dimensions grow only linearly in $d$, for any $g\in \cN_d$ the map 
$$
\bA_i \owns f \mapsto f\cdot g \in \cN_{d+i}
$$
must have a kernel as soon as $i$ is large enough. These are the sought difference 
equations satisfied by $\bK^{-1}$. We will, obviously, have do more work to say 
something more specific about them. 

\subsubsection{}

Note that we have an exact sequence 
\begin{equation}\label{NNL}
  0 \to \cN \to \cN^\sw \to L \to 0 \,,
\end{equation} 
where the third term $L$ satisfies $\dim L_d = d+1$. In fact, $L$ 
is a \emph{line module}, i.e. the module of 
the form 
$$
L = \bA \big/ \bA \, l\,, \quad l\in \bA_1\,.  
$$

\begin{Lemma} In the stable range, i.e.\ for $\sw$ sufficiently far from the 
boundary of $\Omega$, $L$ is a line module with $l=x^\sw \, (x_1+x_2+x_3) \, x^{-\sw}$. 
\end{Lemma}

\begin{proof} Let $g=\bK^{-1}\, x^\sw$ be the generator of $\cN^\sw_0$ and suppose 
$f g \in \cN$ for some $f\in \bA$. This means that $f g$ may be extended
to all of $\Omega$ as a solution of Kasteleyn's equation. In other words, there 
exists a polynomial $g'\in \bA_{\deg f-1}$ such that 
$$
0 = \bK ( f g + g'\, x^{\sw}) = f x^\sw + g' \, x^\sw (x_1 + x_2 + x_3) \,,
$$
proving the assertion. 
\end{proof}

\noindent
{}From this perspective, there isn't much difference between $\cN$ and $\cN^\sw$. 

Somewhat poetically, we will call $\cN$ the \emph{quantum limit shape}. 
Mathematical reasons for this name will be discussed below.

\section{The structure of $\cN$}

\subsection{}

The goal of this Section is to prove the following 

\begin{Theorem}\label{Th1} 
In the stable range and for generic $q$, $\cN$ is generated by $\deg\Omega$ generators of 
degree $1$ subject to $\deg\Omega$ linear relations. In other words, the minimal graded 
free resolution of $\cN$ has the form 
\begin{equation}\label{resN} 
 0 \to \bA(-2)^{\deg\Omega} \to \bA(-1)^{\deg\Omega} \to \cN \to 0 \,.
\end{equation}
Similarly, 
\begin{equation}\label{resNw} 
 0 \to \bA(-2)^{\deg\Omega} \to \bA\oplus\bA(-1)^{\deg\Omega-1} \to \cN^\sw \to 0 \,.
\end{equation}
\end{Theorem}

\noindent
Here we denote, as customary, $\bA(i)_d=\bA_{i+d}$. 

\subsection{}\label{sclass} 

In the commutative case, resolutions 
of the form \eqref{resN} are well known in algebraic geometry, see, in particular, \cite{Beau}. The corresponding 
sheaves on $\Pp$ are of the form $\iota_*\cL$, where
$$
\iota: Q \to \Pp
$$
is an inclusion of a curve of degree $D=\deg\Omega$ and $\cL$ is a line bundle $\cL$ (or a more general torsion-free 
sheaf in case $Q$ is singular) 
of degree $g-1$. Here 
$$
g=(D-1)(D-2)/2
$$ 
is the arithmetic genus of $Q$. Concretely, 
$$
Q = \det R\,, 
$$
where $R$ is the matrix of linear forms that gives the map 
$$
R:  \bA(-2)^D \to \bA(-1)^D
$$
in \eqref{resN}. The condition in \eqref{resN} on 
$\cL$ to have no sections means $\cL \in \Jac_{g-1}(Q)\setminus \Theta$, 
where $\Theta\subset \Jac_{g-1}(Q)$ is the theta divisor. 

The meaning of \eqref{resNw} is parallel, with the difference that $D$ becomes $\deg\Omega+1$ and 
$\deg \cL$ becomes $g$. In the commutative case, \eqref{NNL} implies the support of $\cN^\sw$ has 
the line $l=0$ as a component. 

\subsection{}

One of the eventual goals of the present project is to understand the behavior of $\cN$ as
the mesh size goes to $0$ while $\log q\to 0$ at a comparable rate. The moduli of sheaves of 
the form \eqref{resN} have a natural compactification, which guarantees that any sequence has 
a subsequence converging to an actual sheaf on the usual commutative plane $\Pp$.

It will be shown in \cite{KO} that this limit is supported on the curve $Q$ corresponding 
to the limit shape. This is the reason for calling $\cN$ the quantum limit shape. 

\subsection{}

For any graded $\bA$-module $\cN$, the degrees of its generator, relations, etc., may be read off the 
dimensions of the graded components of the vector spaces 
$$
\Tor_i \cN = \Tor_i (\bA/\bA_{>0},\cN)\,,
$$
where $\bA_{>0}$ is the ideal generated by $\bA_1$. In particular, the existence 
of a free resolution of length 2 is equivalent to the following 

\begin{Lemma} \label{TorN} In the stable range of degrees, $\Tor_i \cN = 0$ for $i>1$. 
\end{Lemma}

\begin{proof} Since $\Tor_{>3}$ vanish identically for the algebra $\bA$, we need to show 
the vanishing for $i=2,3$. Since the resolution \eqref{resM} is defined combinatorially in 
terms of $\Omega$, we conclude $\Tor_i\cM=0$ in the stable range. This is, really, 
the definition of the stable range. Since $\cN \subset \cM$ and $\Tor_4=0$ identically, we conclude
$$
\Tor_3 \cN = 0 \,. 
$$
By construction, there is an exact sequence 
$$
0 \to \cN \to \cM \to \Img \bK \to 0\,, 
$$
and since $\Img \bK \subset \cM$, we have $\Tor_3 \Img \bK = 0$. Now from the long exact 
sequence for $\Tor_i$, it follows that $\Tor_2 \cN =0$. 
\end{proof}

\subsection{}

\begin{Lemma} \label{genN} For generic $q$, $\cN$ is generated by $\cN_1$. 
\end{Lemma}

\begin{proof}
It suffices to consider the commutative case $q_{ij}=1$. Then, on the one hand, $\cN$ is 
annihilated by $x_1+x_2+x_3$, while on the other $\Tor_i \cN = 0$, $i>1$. This means 
$\cN$ corresponds to a vector bundle on the line $x_1+x_2+x_3=0$.  {}From its Hilbert 
polynomial, we see that it must be $\cO(-1)^{\deg \Omega}$, whence the conclusion. 
\end{proof}

\subsection{}
Now it is easy to complete the proof of \eqref{resN}. We have $\deg \Omega = \dim \cN_1$ 
generators in degree 1 and from \eqref{dimNd} we see that they must satisfy $\deg \Omega$
linear relations. There are no other generators by Lemma \ref{genN} and no other 
relations by \eqref{dimNd}. 

The proof of \eqref{resNw} goes along the same lines. Lemma \ref{TorN} still holds 
even though $(\Tor_1 \cM^\sw)_1 \ne 0$. The analog of Lemma \ref{genN} is that $\cN^\sw$
is generated in degrees $0$ and $1$, because in the commutative case $\cN$ corresponds
to the bundle $\cO\oplus \cO(-1)^{\deg \Omega}$ on the line $x_1+x_2+x_3=0$.

It remains 
to explain why the commutative resolution 
$$
0\to \bA(-1)\oplus \bA(-2)^{\deg \Omega} \to \bA\oplus \bA(-1)^{\deg \Omega} \to \cN^\sw \to 0 \,, \quad q_{ij}=1 \,, 
$$
jumps to \eqref{resNw} for generic $q$. In other words, we need to check that the 
generator of $\cN^\sw_0$ no longer satisfies a linear relation for generic $q$. This is 
an easy consequence of the results of the next section. 

Namely, a linear polynomial meets $x_1 x_2 x_3$ in $3$ points, that is, it annihilates
$3$ generators of point modules over $\bA$. As we will see, the annihilator of 
$\cN^\sw_0$ meets each coordinate axis in $\deg\Omega+1$ points.

\section{Boundary points}\label{sbdcN}

\subsection{}

Recall from \cite{Burg} that the curve $Q$ describing the limit shape is determined 
as the unique rational curve of degree $\sdOm$ for which the dual curve $Q^\vee$ is 
inscribed in $\Omega$. This means, in particular, that $Q$ meets each 
coordinate line of $\Pp$ in $\sdOm$ specified points.

In this section, we 
will see that the quantum limit shape $\cN$ satisfies the exact noncommutative
analog of this incidence. 

\subsection{}

We define 
$$
\partial_3 \cN = \cN \big/ x_3 \cN \,.
$$
We will see that $\partial_3 \cN$ is a direct sum of $\deg\Omega$ point modules 
labeled by the horizontal boundaries of $\Omega$ as in Figure \ref{Fig6}.

By definition, the Hilbert polynomial of \emph{a point module} is equal to the constant $1$. 
Up to modules of finite length, point modules are parametrized by the toric divisor 
$x_1 x_2 x_3 = 0$ of $\Pp$. The correspondence is a follows 
$$
(a_1 : a_2 : 0) \leftrightarrow \bA\big/\lang x_3, a_2 x_1 - a_1 x_2 \rang \,.
$$
The ratios  $a_2/a_1$ for the summands of $\partial_3 \cN$ will be determined 
by the vertical coordinate of the horizontal boundaries of $\Omega$. 

\subsection{}

The Hilbert function evaluation 
$$
\dim \left(\partial_3 \cN\right)_d = \deg \Omega \,, \quad d>0\,,
$$
is a consequence of the following

\begin{Lemma}\label{noker}
Left multiplication by $x_3$ has no kernel acting on $\cN$. 
\end{Lemma}

\begin{proof}
A polynomial in the kernel of left multiplication by $x_1$ 
has a support in a strip of width 1 along the black boundaries, as in Figure \ref{Fig6}, right. 
It is impossible for such function to be annihilated by the 
Kasteleyn operator. 
\end{proof}

\subsection{White boundaries}\label{swbdry}

Let $\Omega'$ be a nontileable domain obtained by moving one of the white 
horizontal boundaries of $\Omega$ one step in. Denote by $\cM'$ the corresponding 
monomial module. Let $\bK':\cM'\to\cM'$ be right multiplication by 
$x_1+x_2+x_3$ and let $\cN'=\cN \cap \cM'$ be its kernel. Since $x_3 \cN \subset \cM'$, 
$\partial_3 \cN$ surjects onto $\cN/\cN'$. 

By cutting white boundary strips off $\Omega'(i)$, $i>1$, we see as in the proof of Lemma \ref{lsurK}
that $\bK'$ surjects onto $\cM'_i$ for $i>1$. Since $\ind \bK'_0 = -1$, it follows 
from Lemma \ref{lemdimM} that $\cN/\cN'$ is a point module. 

\begin{Lemma} \label{Lwhbd}
Let $a$ be the vertical coordinate of the strip $\Omega\setminus\Omega'$, that is, let $\deg_{x_3} f = a$
for all $f\in \cM/\cM'$.  Then
$$
\cN/\cN' \cong \bA/\lang x_3, x_1 + q^{a}\,x_2 \rang \, (1) \,,
$$
starting in degree $1$. 
\end{Lemma}

\begin{proof} The series
$$
(x_1+x_2)^{-1} = x_1^{-1} - x_1^{-1} x_2 x_1^{-1} + x_1^{-1} x_2 x_1^{-1} x_2 x_1^{-1} - \dots 
$$
gives $1$ after left or right multiplication by $(x_1+x_2)$. Interchanging the roles of $x_1$ and $x_2$ and
taking the difference 
$$
\delta = (x_1+x_2)^{-1} - (x_1+x_2)^{-1}
$$
we get an analog of the usual $\delta$-function, which is annihilated by both left and right multiplication 
by $(x_1+x_2)$. 

An element of $\left(\cN/\cN'\right)_d$ is a truncation of the series $x_3^a x_1^{-a+d+1} \delta$ on 
both sides. Left multiplication by 
$$
x_1 + q^a q_{21}^{d+1} x_2 
$$
annihilates it, whence the conclusion. 
\end{proof}

\subsection{Black boundaries}\label{sbbdry}

Now let $\Omega'$ be a nontileable domain obtained by moving one of the black
horizontal boundaries of $\Omega$ one step out. We have a map $\cM'\to \cM$ corresponding 
to the restriction of functions and from Lemma \ref{noker} we conclude that 
it yields an injection $\cN'\to \cN$. Further, $x_3 \cN$ is in the image of $\cN'$,
hence $\cN/\cN'$ is again a point module onto which $\partial_3 \cN$ surjects. 

\begin{Lemma} \label{Lblbd}
Let $a$ be the vertical coordinate of the strip $\Omega'\setminus\Omega$, that is, let $\deg_{x_3} f = a$
for all $f$ in the kernel of the restriction map $\cM'\to \cM$.  Then
$$
\cN/\cN' \cong \bA/\lang x_3, x_1 + q^{a}\,x_2 \rang \, (1) \,,
$$
starting in degree $1$. 
\end{Lemma}

\begin{proof}
Let $f$ be in $\cN_d$ and denote by $f_{d-a+1} \, x_3^{a-1}$ the monomials in $f$ along the
boundary in question. Here $f_{d-a+1}$ denotes a polynomial in $x_1$ and $x_2$ of degree 
$d-a+1$. Clearly, $f$ may be extended to an element in $\cN'$ if and only if we can find a polynomial 
$g_{d-a}(x_1,x_2) \, x_3^a$ with support in $\Omega'\setminus \Omega$ such that
$$
f_{d-a+1} \, x_3^a = g_{d-a} \, x_3^a \, (x_1+x_2) \,.
$$
The commutation relations in $\bA$ imply that for any $f_k\in \bA_k$ which does not depend on $x_3$
and any $C$ we can find $f'_k\in \bA_k$ such that 
$$
(x_1 + C\, x_2 ) \,  f_k = f'_k \, (x_1 + q_{12}^k \, C\, x_2) \,.
$$
{}From this it follows that $x_1 + q^{a}\, q_{21}^{d+1} \,x_2$ annihilates $\left(\cN/\cN'\right)_d$,
as was to be shown. 
\end{proof}

\subsection{}

We can summarize the discussion as follows. 

\begin{Theorem} We have 
\begin{equation}\label{bdryN}
\partial_3 \cN \cong \bigoplus_{i=1}^{\deg \Omega} \bA/\lang x_3, x_1 + q^{a_i}\,x_2 \rang \, (1) \,,
\end{equation}
starting in degree $1$, where $a_i$ are the 
heights of the horizontal boundaries of $\Omega$ as above. 
\end{Theorem}

\begin{proof}
 The Hilbert polynomial equals 
the constant $\deg\Omega$ on both sides.
We constructed a map from the LHS to the RHS in \eqref{bdryN}. Its kernel consists of functions $f$
that vanish on the boundary strips along the white boundaries of $\Omega$ and may be extended as solutions of
$\bK f =0$ to a horizontal strip just beyond the black boundaries of $\Omega$. This modified domain is a
translate of $\Omega(-1)$ in the $x_3$ direction and our conditions on $f$ imply $f\in x_3 \cN$. 
Thus the map above is an isomorphism. 
\end{proof}

\section{Correspondences and renormalization}\label{srenorm} 

\subsection{}

Let $\Omega'$ be obtained from $\Omega$ by moving one boundary segment by one step, as in 
Sections \ref{swbdry} and \ref{sbbdry}. We found that the corresponding modules $\cN$ and $\cN'$
fit into an exact sequence of the form 
\begin{equation}\label{exactP}
 0\to \cN' \to \cN \to P \to 0
\end{equation}
with a point module $P$. Modules of the form \eqref{resN} form an open set in the moduli spaces
of $\bA$-modules of rank $0$ and 
$$
(c_1,c_2) = (D,D(D+3)/2)\,, \quad D=\deg \Omega \,.
$$
Let $\Ms(c_1,c_2)$ denote a finite cover of this open set over which the ordering of
the summands in \eqref{bdryN} is chosen.

For general $c_2$, we define $\Ms(c_1,c_2)$ as 
the moduli space of $\bA$-modules with the same minimal resolution as the generic 
commutative resolution plus an ordering of boundary points. For example, for 
$$
(c_1,c_2) = (D,D(D+3)/2-k)\,, \quad 0\le k \le D/2\,,
$$
the corresponding modules $\cN''$ are of the form
$$
0\to \bA(-2)^{D-k} \to \bA^k \oplus\bA(-1)^{D-k} \to \cN'' \to 0 \,, 
$$
generalizing \eqref{resNw}. 

\subsection{Noncommutative shift on the Jacobian}

It is easy to see that 
$$
\dim \Ms(c_1,c_2) = c_1^2 + 1\,,
$$
and, in particular, it doesn't depend on $c_2$.

\begin{Lemma}
The correspondence 
$$
\{ (\cN',\cN) \} \subset \Ms(c_1,c_2+1) \times \Ms(c_1,c_2) 
$$
defined by \eqref{exactP} is the graph of a birational map. 
\end{Lemma}

\begin{proof}
It suffices to consider the commutative case, when this becomes a shift by $P$ on 
the Jacobian of the curve $Q$, see Section \ref{sclass}. 
\end{proof}

This means that we have an action of of a group $\bS\cong\Z^{3 c_1}$ on 
$
\bigsqcup_{c_2} \Ms(c_1,c_2)
$
by birational transformations. A subgroup of the form $\Z^{3 c_1-1}$ preserves $c_2$ and acts birationally 
on individual components. Perhaps the title of this subsection is the appropriate name for this
group action. 

Parallel group actions may be defined for other noncommutative surfaces. They turn out to encompass 
several previously studied discrete dynamical systems. This is the subject of a joint work by E.~Rains 
and the author, the results of which will appear in \cite{RO}.

\subsection{}

Renormalization is a central concept in mathematical physics. For any tileable domain $\Omega$ and 
any $m\in\Z_{>0}$, the scaled domain $m\Omega$ is again tileable and it is natural to ask how 
this scaling transformation affects our random surfaces. Equivalently, of course, one can keep $\Omega$
fixed and divide the mesh size by $m$. 

To quantify the word ``affects'', one tries to summarize 
the behavior of random surfaces in terms of finitely many essential degrees of freedom. One further 
hopes to define an action of the group $\R_{>0}$ on this space extending the scaling transformations
above. 

While this procedure is a very powerful guiding principle, in practice one usually has to use 
various approximations to make it work. This is only natural since one is trying to squeeze
an infinite-dimensional problem into a finite-dimensional dynamical system. 

\subsection{}

Our kind of problems are special in that they have a certain built-in finite-dimensionality, 
starting from a finite number segments that bound $\Omega$. The quantum limit shape $\cN$ 
also varies in a finite-dimensional moduli space. Since $\cN$ encodes the essential information 
about the correlation functions, the renormalization dynamics may be considered
understood once the scaling action on $\cN$ is determined. 

Clearly, scaling transformations are composed out of many noncommutative shifts on the Jacobian,
and more precisely we have 

\begin{Theorem} The scaling $\Omega\mapsto m\cdot\Omega$ preserves the $\bS$-orbit of 
$\cN$ and 
intertwines the action of $s\in \bS$
with the action of $m\cdot s\in \bS$. 
\end{Theorem}

\subsection{}

If the parameter $q$ is adjusted simultaneously, then the limit $m\to\infty$ becomes the \emph{thermodynamic} limit
$$
\textup{mesh}\to 0, \quad \log q = O(\textup{mesh}) \,,
$$
which is the limit that we were planning to take all along. In this limit, noncommutative shifts of the Jacobian 
may be viewed as a perturbation of the commutative shift, thus joining a much-studied area of 
perturbations of integrable systems. Their analysis from this point of view will appear in \cite{KO}. 

In particular, it will be shown in \cite{KO} that, indeed the quantum limit shape $\cN$ is a deformation of 
the curve $Q$ that defines the (classical) limit shape.

\section{Outlook}

\subsection{}

The Kasteleyn operator on $\Gamma_6$ has an infinite-dimensional kernel and the Kasteleyn equation 
needs to be supplemented by boundary conditions in order to have a unique solution. By contrast, 
once a second difference equation of degree $d$ is known, the solutions form a $d$-dimensional linear space, spanned by modulated plane waves. 
This simple principle gives a powerful way to control $\bK^{-1}$ in the thermodynamic limit which 
is the basic analytic issue in the analysis of stepped surfaces. 

In fact, optimistically, one may expect these techniques to overcome the difficulties that lie
in the way of proving the CLT for stepped surfaces with polygonal boundaries (see \cite{Ken} for techniques
that can handle a different sort of boundary conditions) as well determining the local correlations. Further, 
since polygonal boundaries are dense in the space of all boundaries, the calculation of 
local correlations for them has direct implication to the classification of Gibbs measures (about which 
\cite{Shef} contains a wealth of information). 

\subsection{}

Kasteleyn theory and the formalism of \cite{Burg} work for any periodic 
bipartite planar dimer. It would be very interesting to study the quantum limit shape in this generality. 
It would be also very interesting to find applications to difference equations other than the Kasteleyn 
equation, such as e.g.\ discrete Dirac equations in dimensions $>2$.

\subsection{}

For very special kinds of boundary conditions (see e.g.\ \cite{EC,Seattle} for an introduction) the $q$-weighted 
stepped surface partitions functions $Z$ become generating functions $Z_{DT}$ for the 
Donaldson-Thomas invariants of toric CY three-folds\footnote{The customary parameter $q$ in DT theory differs from ours by a minus sign.}.

Very generally, for any smooth projective three-fold $X$ the expansion of $\log Z_{DT}$ in powers of $\log q$ was conjectured to generate the Gromov-Witten 
invariants of the same three-fold $X$, genus by genus  \cite{mnop}. For a toric three-fold $X$, this conjecture was proven in \cite{moop}. 

In particular, this relates the genus $0$ Gromov-Witten invariants of $X$ to the leading asymptotics
of $\log Z$ as $\log q\to 0$, and hence to the limit shape $Q$. \emph{Mirror symmetry}, which is certainly too complex
and multifaceted a phenomenon to be discussed here with any precision, associates
basically the same curve $Q$ to $X$. 
Thus the limit shape point of view puts mirror symmetry on a firm probabilistic ground in this particular instance. 

The quantum limit shape $\cN$ captures the fluctuations and hence all 
orders of the expansion of $\log Z$. This makes it a strong candidate for the as yet mysterious higher-genus 
mirror of $X$. Moreover, the quantum limit shape $\cN$, being a categorical object, might 
stand a better chance of generalization to non-toric $X$ than the underlying box-counting. 
Progress in this direction remains both very desirable and scarce.

\subsection{}
There exists a different (and, at present, conjectural) way to extract the 
higher genus GW invariants out of $Q$, which was proposed in \cite{BKMP} based
on the diagrammatic techiques developed in the random matrix context, 
see in particular \cite{CEO}. The use of noncommuting variables in a related 
context was advocated, in particular, in \cite{DHSV}. 

It is reasonable
to expect the two approaches to converge, especially since various 
random matrix models may naturally be viewed as continuous limits of 
stepped surfaces. For example, one sees random matrices quite 
directly near the points where $Q$ intersects the coordinate
 axes \cite{birth}. 
There are also numerous parallels between random matrices and 
Plancherel-like measures on partitions, which are induced on slices
of stepped surfaces \cite{uses}. 

A more ambitious goal may be to push the theory away from the 
$K_X=0$ case.  As a first problem in the $K_X\ne 0$ direction,
one can try the equivariant vertex \cite{mnop,moop}. As a random surface 
model, it is very nonlocal and otherwise distant from what is 
perceived as natural in statistical mechanics. To its credit, 
it has a map, due to Nekrasov \cite{Solvay}, onto certain local 2-dimensinal
lattice fermions. In contrast to Kasteleyn theory, these 
fermions are now interacting and it remains to be seen how 
much progress one can make in this more general setting.


\begin{thebibliography}{99}



\bibitem{Beau}
A.~Beauville, 
\emph{Determinantal hypersurfaces}, Mich.\ Math.\ J.\ \textbf{48}, 39--64 (2000). 


\bibitem{BKMP}
V.~Bouchard, A.~Klemm, M.~Marino, S.~Pasquetti,
\emph{Remodeling the B-model},
  Comm.\ Math.\ Phys.\  \textup{287}  (2009),  no.\ 1, 117--178. 

\bibitem{CEO}
L.~Chekhov, B.~Eynard, N.~Orantin, 
\emph{Free energy topological expansion for the 2-matrix model}.
  J.\ High Energy Phys.\  2006,  no.\ 12.



\bibitem{CKP}
Cohn, H., Kenyon, R., Propp, J., 
A variational principle for domino tilings,
\emph{Journal of AMS}, {\bf 14}(2001), no.~2, 297-346.



\bibitem{DHSV}
R.~Dijkgraaf, L.~Hollands, P.~Sulkowski, C.~Vafa, 
\emph{Supersymmetric gauge theories, intersecting branes and free fermions},
 J.\ High Energy Phys.\  2008,  no.\ 2. 


\bibitem{Kas}
P.~Kasteleyn, 
\emph{Graph theory and crystal physics}, 
Graph Theory and Theoretical Physics,  43--110, 
Academic Press, 1967


\bibitem{Ken}
R.~Kenyon, 
\emph{Height fluctuations in honeycomb dimers}, 
\texttt{math-ph/0405052}.

\bibitem{KenLec}
R.~Kenyon, \emph{Lectures on dimers}, available from 
\texttt{http://www.math.brown.edu/$\sim$rkenyon/papers/dimerlecturenotes.pdf}

\bibitem{Burg}
R.~Kenyon and  A.~Okounkov, 
\emph{Limit shapes and complex Burgers equation},
\texttt{math-ph/0507007}. 


\bibitem{KO}
I.~Krichever and A.~Okounkov, 
in preparation. 

\bibitem{mnop} D.~Maulik, N.~Nekrasov, A.~Okounkov, and R.~Pandharipande, 
\emph{Gromov-Witten
theory and Donaldson-Thomas theory}, I. \& II., 
\texttt{math.AG/0312059}, \texttt{math.AG/0406092}.

\bibitem{moop} D.~Maulik, A.~Oblomkov, A.~Okounkov, and R.~Pandharipande, 
\emph{Gromov-Witten/Donaldson-Thomas correspondence for toric 3-folds}, 
\texttt{arXiv:0809.3976}. 

\bibitem{Solvay}
N.~Nekrasov, 
\emph{Topological strings and two dimensional electrons},
The Quantom Structure of Space and Time, 
Proceedings of the 23rd Solvay Conference on Physics, 
edited by D.~Gross, M.~Henneaux, A. Sevrin,
World Scientific, 2007. 

\bibitem{Takagi}
N.~Nekrasov, \emph{Instanton partition functions and M-theory}, Vth Takagi Lectures,  Japan. J. Math. 4, 63-93 (2009) 


\bibitem{uses}
A.~Okounkov, 
\emph{The uses of random partitions},
XIVth International Congress on Mathematical Physics,  
379--403, World Sci., 2005. 

\bibitem{EC}
A.~Okounkov, 
\emph{Random surfaces enumerating algebraic curves}, 
Proceedings of Fourth European Congress of 
Mathematics, EMS, 751--768, 
\texttt{math-ph/0412008}.  


\bibitem{Seattle}
A.~Okounkov,
\emph{Geometry and physics of localization sums}, 
\texttt{http://www.math.columbia.edu/~thaddeus/seattle/okounkov.pdf}. 



\bibitem{birth}
A.~Okounkov, 
\emph{The birth of a random matrix}, 
Mosc.\ Math.\ J.\ \textbf{6} (2006), no.\ 3, 553--566. 




\bibitem{RO}
E.~Rains and A.~Okounkov, 
in preparation. 


\bibitem{Shef}
S.~Sheffield, 
\emph{Random surfaces}, Ast\'erisque \textbf{304} (2005). 

\bibitem{SV}
J.\ T.\ Stafford and  M.\ Van den Bergh,
\emph{Noncommutative curves and noncommutative surfaces}, 
Bull.\ AMS \textbf{38} (2001), 171--216. 



\end{thebibliography}
\end{document}